\documentclass[12pt]{amsart}
\usepackage{amsmath}
\usepackage{amssymb}

\makeatletter
\@namedef{subjclassname@2020}{\textup{2020} Mathematics Subject Classification}
\makeatother

\newcommand{\s}{\sigma}

\newcommand{\ov}{\overline}

\newtheorem{theorem}{Theorem}%[section]
\newtheorem{corollary}{Corollary}

\newtheorem{lemma}{Lemma}
\newtheorem{proposition}{Proposition}

\theoremstyle{definition}
\newtheorem{definition}{Definition}
\newtheorem{remark}{Remark}

\newcommand{\beq}{\begin{equation}}
\newcommand{\eeq}{\end{equation}}
\newcommand{\beqs}{\begin{eqnarray*}}
\newcommand{\eeqs}{\end{eqnarray*}}
\newcommand{\beqn}{\begin{eqnarray}}
\newcommand{\eeqn}{\end{eqnarray}}
\newcommand{\beqa}{\begin{array}}
\newcommand{\eeqa}{\end{array}}

\usepackage{amsmath}
\numberwithin{equation}{section}
\numberwithin{theorem}{section}
\numberwithin{lemma}{section}
\numberwithin{remark}{section}
\numberwithin{assumption}{section}
\numberwithin{definition}{section}
\numberwithin{fact}{section}
\numberwithin{question}{section}
\numberwithin{proposition}{section}

\begin{document}

\title{Second order estimates for $\chi$-semi convex solutions of Hessian equations on Hermitian manifolds}

\author{Xiaojuan Chen}
\address{Institute of Mathematics, Hunan University, Changsha 410082, China}
\email{cxj@hnu.edu.cn}

\author{Qiang Tu}
\address{Faculty of Mathematics and Statistics, Hubei Key Laboratory of Applied Mathematics, Hubei University,  Wuhan 430062, P.R. China}
\email{qiangtu@hubu.edu.cn}

\author{Ni Xiang}
\address{Faculty of Mathematics and Statistics, Hubei Key Laboratory of Applied Mathematics, Hubei University,  Wuhan 430062, P.R. China}
\email{nixiang@hubu.edu.cn}

%\thanks{This research was supported by funds from Natural Science Foundation of Hubei Province, China,
%No.2020CFB606, and Natural Science Foundation of China No.11971157.}

%\date{}

\begin{abstract}
In this paper, we establish the modified concavity inequality for complex Hessian equations
 under the semi-convexity assumption  inspired by Lu \cite{Lu23} and Zhang \cite{Z24} for real case. Then second order estimates for admissible solutions of complex Hessian equations on compact Hermitian manifolds with both sides of equations depending on gradient terms are obtained by taking advantage of the crucial inequality.

%We establish second order estimates for admissible solutions of complex Hessian equations on compact Hermitian manifolds with both sides of equations depending on gradient terms. Currently, the convexity assumption of solution plays a significant role in the research on such issues. In this paper, we are able to weaken the convexity assumption to the semi-convex assumption based on a concavity inequality.
\end{abstract}

\thanks{Research of Chen was supported by the Postdoctoral Fellowship Program of CPSF under Grant Number GZB20250702;
 Research of Tu and Xiang was supported by funds from
%Hubei Provincial Department of Education
%Key Projects D20181003, Natural Science Foundation of Hubei Province, China, No. 2020CFB246
the National Natural Science Foundation of China No. 11971157, 12101206, 12426532.}
%\thanks{$\ast$ Corresponding author}

\maketitle
\vskip4ex

%%%%%%%%%%%%%%%%%%%%%%%%%%%%%%%%%%%
\section{Introduction}
Let $(M,\omega)$ be a compact complex manifold of complex dimension $n\geq2$ with Hermitian metric $\omega$. For any $u\in C^{\infty}(M)$, we consider a smooth real $(1,1)$-form $\chi'(z,u)$ satisfying the positivity condition $\chi'\geq\varepsilon\omega$ for some $\varepsilon>0$. Given any smooth $(1,0)$-form $a$ on $M$, we obtain a new real $(1,1)$-form
\begin{eqnarray}\label{ch}
\chi(z,u)=\chi'(z,u)+\sqrt{-1}a\wedge\overline{\partial}u-\sqrt{-1}\overline{a}\wedge\partial u+\sqrt{-1}\partial\overline{\partial} u.
\end{eqnarray}
In this paper, we study second order estimates for the following complex Hessian equation
\begin{eqnarray}\label{K-eq}
\chi^{k} \wedge \omega^{n-k} =\psi(z, Du, u) \omega^n
\end{eqnarray}
for $1\leq k\leq n$ since the a priori estimates is a critical step in the use of the continuity method, which is a natural approach to solve equation \eqref{K-eq}. Here $\psi(z, v, u) \in C^{\infty} \left((T^{1,0} (M))^{\ast} \times \mathbb{R}\right)$ is a given positive function, $D$ is the covariant derivative with respect to the metric $\omega$.

The complex Hessian equation \eqref{K-eq} without gradient term on both sides (i.e., $\chi=\chi'+\sqrt{-1}\partial\overline{\partial}u$, $\psi=\psi(z,u)$) has been extensively studied due to its geometric applications such as $J$-flow (see \cite{SW08}) and quaternionic geometry (see \cite{AV10}). Li \cite{Li04} first considered the Dirichlet problem for complex Hessian equations on domain in $\mathbb{C}^n$, then by Blocki \cite{B05}. On compact K\"{a}hler manifolds, Hou \cite{H09} proved the existence of a smooth admissible solution of equation \eqref{K-eq} when $\chi'=\omega,a=0$ and $\psi=\psi(z)$ by assuming the nonnegativity of the holomorphic bisectional curvature of $\omega$. The curvature assumption was subsequently removed by Hou-Ma-Wu \cite{HMW10} and Dinew-Kolodziej \cite{DK14}. On compact Hermitian manifolds, Cherrier \cite{Ch}, Hanani \cite{Ha}, Guan-Li \cite{GL09}, Zhang \cite{Z10}, Tosatti-Weinkove \cite{TW10} and Zhang-Zhang \cite{ZZ11} studied complex Monge-Amp\`{e}re equations. For complex Hessian equations or more general fully nonlinear elliptic equations, the related results can refer to Sun \cite{S14}, Zhang \cite{Z17}, Kolodziej-Nguyen \cite{KN16}, Sz\'{e}kelyhidi \cite{SZ18} and the references therein.

Furthermore, complex Hessian equations with gradient terms on the left hand side arise from the Gauduchon conjecture, which was studied by Sz\'{e}kelyhidi-Tosatti-Weinkove \cite{STW17}, and see also Guan-Nie \cite{GN23} for related work. Complex Monge-Amp\`{e}re equations with gradient terms inside the determinant can be found in the study of the Calabi-Yau equation on certain symplectic non-K\"{a}hler 4-manifolds, see Fino-Li-Salamon-Vezzoni \cite{FLSV13}. The related work can also see Tosatti-Weinkove \cite{TW191}, where they solved complex Monge-Amp\`{e}re equations with gradient terms left open by Yuan \cite{Yuan18}. More results for complex Hessian equations or general fully nonlinear elliptic equations with linear gradient terms on the left hand side can refer to Zhang \cite{Z10}, Feng-Ge-Zheng \cite{FGZ19} and Yuan \cite{Yuan18,Yuan20}.

It is worth noting that in the above mentioned works, the right hand side function $\psi$ is independent of $Du$. In fact, equation \eqref{K-eq} with $\psi=\psi(z,Du,u)$ has been relatively less studied. An important case corresponding to $k=n=2$ (i.e., complex Monge-Amp\`{e}re equations in two dimensions) was solved by Fu-Yau \cite{FY07, FY08} of a Strominger system on a toric fibration over a K3 surface. For general dimension $n$, the corresponding results on compact K\"{a}hler manifolds or Hermitian manifolds can refer to Phong-Picard-Zhang \cite{PPZ16, PPZ17, PPZ19, PPZ21}, Chu-Huang-Zhu \cite{CHZ19, CHZ2019}.
Furthermore, complex Hessian equations with gradient terms on both sides were studied by Dong-Li \cite{DL19} on Hermitian manifolds, and second order estimates were obtained for $\chi\in \Gamma_n(M)$. Subsequently, Dong \cite{D21, Do21} proved second order estimates under a weaker convexity assumption $\chi\in\Gamma_{k+1}(M)$ for  equation \eqref{K-eq} and a class of general Hessian equations on Hermitian manifolds, respectively.

From analysis point of view, a natural problem is weather we can weaken the assumption $\chi\in\Gamma_{k+1}(M)$ in \cite{D21} to a more natural semi-convex assumption inspired by Lu \cite{Lu23} and Zhang \cite{Z24}. Now we define $\chi$-semi convex function on Hermitian manifolds, which generalizes the definition of semi-convex function for real case.

\begin{definition}
Let $h\in \mathcal{A}^{1,1}(M)$ be a smooth real $(1,0)$-form on $(M,\omega)$. It can be written as $h=\sqrt{-1}h_{i\overline{j}}dz^i\wedge d\overline{z}^j$ in a local coordinate. We say $h$ is semi-convex, if there exists a constant $A>0$ such that
$$\lambda_k(h_{i\overline{j}})>-A, ~1\leq k\leq n,~ \forall~x\in M.$$
And we call the function $u$ is $\chi$-semi convex if $\chi$  defined in \eqref{ch} is semi-convex.
\end{definition}

\begin{definition}
A function $u\in C^4(M)$ is called an admissible solution of equation \eqref{K-eq} if $\chi\in \Gamma_k(M)$. The G$\mathring{\text{a}}$rding's cone on $M$ is defined by
$$\Gamma_k(M)=\{\chi\in \mathcal{A}^{1,1}(M):\sigma_i(\chi)>0,~\forall~1\leq i\leq k\},$$
where $\mathcal{A}^{1,1}(M)$ is the space of real smooth $(1,1)$-forms on Hermitian manifold $(M,\omega)$.
\end{definition}

The main theorem is as follows.
\begin{theorem}\label{main}
Let $(M,\omega)$ be a compact Hermitian manifold of complex dimension $n$. Suppose $\chi'(z,u)\geq \varepsilon\omega$ and $u\in C^{\infty}(M)$ is an admissible, $\chi$-semi convex solution of equation \eqref{K-eq}. Then we have uniform second order estimates
$$|D\overline{D}u|\leq C,$$
where $C$ is a uniform constant depending on $(M, \omega),n,k,\varepsilon,\chi',\psi,a,\sup_M|u|$, $\sup_M|Du|$.
\end{theorem}

It is easy to see that $\sigma_{k+1}(\chi)>-A$ for some positive constant $A$ implies the $\chi$-semi convex condition according to Lemma 2.2 in \cite{Z24}.
\begin{corollary}
Let $(M,\omega)$ be a compact Hermitian manifold of complex dimension $n$. Suppose $\chi'(z,u)\geq \varepsilon\omega$ and $u\in C^{\infty}(M)$ is an admissible solution of equation \eqref{K-eq} with $\sigma_{k+1}(\chi)>-A$ for some positive constant $A$. Then we have uniform second order estimates
$$|D\overline{D}u|\leq C,$$
where $C$ is a uniform constant depending on $(M, \omega),n,k,\varepsilon,\chi',\psi,a,\sup_M|u|$, $\sup_M|Du|$.
\end{corollary}

Due to the presence of $Du$ in both sides of equation \eqref{K-eq}, the prime challenge in establishing second order estimates is managing terms such as $|DDu|^2$, $|D\overline{D}u|^2$ and some bad third order terms when differentiating the equation twice. However, we cannot directly control bad terms as in the real case owing to complex conjugacy. Furthermore, in comparison to calculations on K\"{a}hler manifolds, a greater number of undesirable third order terms in the form of $T\ast D^3u$ are produced, here $T$ is the torsion of $\omega$ on Hermitian manifolds.  To surmount these obstacles, it is essential  to establish a concavity inequality. Lu \cite{Lu23} established curvature estimates for semi-convex solutions of real Hessian equations based on the following concavity inequality
%For the more general right hand function $\psi=\psi(z,Du,u)$, the dependence of $Du$ creates more difficulties. The main obstacle of establishing second order estimates is to handle the terms such as $|DDu|^2$, $|D\overline{D}u|^2$ and some bad third order terms when one differentiates the equation twice. Actually, the linear gradient terms in $\chi$ bring some bad third order terms and the differences between the real case and the complex case make it more difficult to control the negative third order terms due to complex conjugacy. During the calculation process on Hermitian manifolds, more bad third order terms of the form $T\ast D^3u$ will be generated, where $T$ is the torsion of $\omega$. In order to overcome these difficulties, the key is to establish a concavity inequality. Lu \cite{Lu23} established curvature estimates for semi-convex solutions of real Hessian equations based on the following concavity inequality
\begin{equation}\label{Lu}
  -\sum_{p\neq q}\sigma_k^{pp,qq}\xi_p\xi_q+\frac{(\sum_i\sigma_k^{ii}\xi_i)^2}{\sigma_k}
  +\delta_0\sum_{i>l}\frac{\sigma_k^{ii}\xi_i^2}{\lambda_1}\geq(1-\epsilon)\frac{\sigma_k\xi_1^2}{\lambda_1^2},
\end{equation}
%\begin{equation}\label{Lu}
%  -\sum_{p\neq q}\frac{\sigma_k^{pp,qq}\xi_p\xi_q}{\sigma_k}+\frac{(\sum_i\sigma_k^{ii}\xi_i)^2}{\sigma_k^2}
 % +\delta_0\sum_{i>l}\frac{\sigma_k^{ii}\xi_i^2}{\lambda_1\sigma_k}\geq(1-\epsilon)\frac{\xi_1^2}{\lambda_1^2},
%\end{equation}
where $\xi=(\xi_1,\cdots,\xi_n)$ is an arbitrary vector in $\mathbb{R}^n$.
%The concavity inequality and semi-convex assumption enable us to handle third order terms. But at the same time, it will also produce the term $\sigma_k(\lambda|1)$, which is hard to handle without the assumption $\chi\in\Gamma_{k+1}(M)$.
Recently, Zhang \cite{Z24} improved the
concavity inequality \eqref{Lu} and derived the following Lemma.
\begin{lemma}[\cite{Z24}]
Assume $\{\lambda_i\}\in\Gamma_k$, $\lambda_1\geq\lambda_2\geq\cdots\geq\lambda_n$ and $\lambda_n>-A$ for a constant $A>0$. Then there exists some constant $C>0$ depending on $n,k,\sigma_k,A$ and small constant $\delta_0>0$ depending on $k$ such that if $\lambda_1\geq C$, then the following inequality holds
\begin{equation}\label{Zhang}
-\sum_{p\neq q}\sigma_k^{pp,qq}\xi_p\xi_q+K\frac{(\sum_i\sigma_k^{ii}\xi_i)^2}{\sigma_k}+2\sum_{i>1}
\frac{\sigma_k^{ii}\xi_i^2}{(\lambda_1+A+1)}\geq(1+\delta_0)\frac{\sigma_k^{11}\xi_1^2}{\lambda_1}
\end{equation}
%\begin{equation}\label{Zhang}
%-\sum_{p\neq q}\frac{\sigma_k^{pp,qq}\xi_p\xi_q}{\sigma_k}+K\frac{(\sum_i\sigma_k^{ii}\xi_i)^2}{\sigma_k^2}+2\sum_{i>1}
%\frac{\sigma_k^{ii}\xi_i^2}{(\lambda_1+A+1)\sigma_k}\geq(1+\delta_0)\frac{\sigma_k^{11}\xi_1^2}{\lambda_1\sigma_k}
%\end{equation}
for some sufficiently large $K>0$ (depending on $\delta_0$), where $\xi=(\xi_1,\cdots,\xi_n)$ is an arbitrary vector in $\mathbb{R}^n$.
\end{lemma}

By comparing the above two inequalities \eqref{Lu} and \eqref{Zhang}, each of these two inequalities has its own advantages. Specifically, we find that the term $\delta_0\sum_{i>l}\frac{\sigma_k^{ii}\xi_i^2}{\lambda_1}$ in \eqref{Lu} can be easier to control than the term $2\sum_{i>1}\frac{\sigma_k^{ii}\xi_i^2}{(\lambda_1+A+1)}$ in \eqref{Zhang} since $\delta_0$ can be chosen arbitrarily small. However, $\frac{\sigma_k\xi_1^2}{\lambda_1^2}$ in \eqref{Lu} will lead to the bad  term $\sigma_k(\lambda|1)$, which is unmanageable term without the assumption $\chi\in \Gamma_{k+1}(M)$.
It seems that the term $(1+\delta_0)\frac{\sigma_k^{11}\xi_1^2}{\lambda_1\sigma_k}$ in \eqref{Zhang} is a higher choice. For complex Hessian equations, some good third order terms will be lost due to complex conjugacy. Therefore, we will fully absorb the processing techniques of the above two inequalities, and optimize the coefficients in \eqref{Zhang} from $2$ to $1-\varepsilon_0$ for some small $\varepsilon_0>0$ while retaining the good terms.
Then we establish the modified concavity inequality.
\begin{lemma}\label{key lemma}
Let $W=(\omega_{p\overline{q}})$ be a Hermitian tensor and its eigenvalues $\lambda(W)=(\lambda_1,\cdots,\lambda_n)\in \Gamma_k$ satisfying $\lambda_1\geq\cdots\geq\lambda_n>-A$ for a constant $A>0$. For any $\varepsilon_0\in(0,1)$, if $\lambda_1$ sufficiently large, then for $j=1,\cdots,n$ and some positive constant $K$ depending only on $k$, the following inequality holds
\begin{equation}\label{key}
  -\sum_{p\neq q}\frac{\sigma_k^{p\overline{p},q\overline{q}}\omega_{p\overline{p}j}\omega_{q\overline{q}\overline{j}}}{\sigma_k}+K\frac{|D_j\sigma_k|^2}{\sigma_k^2}
  +(1-\varepsilon_0)\sum_{i>1}\frac{\sigma_k^{i\overline{i}}|\omega_{i\overline{i}j}|^2}{\lambda_1\sigma_k}\geq(1-\varepsilon_0)
  \frac{\sigma_k^{1\overline{1}}|\omega_{1\overline{1}j}|^2}{\lambda_1\sigma_k}.
\end{equation}
\end{lemma}
\begin{remark}
The selection of auxiliary functions is crucial in dealing with third order terms. Dong \cite{D21} employed all eigenvalues of the solution to obtain additional good third order terms to control  bad terms. In our test function, we can choose the largest eigenvalue as the primary term with the help of the concavity inequality \eqref{key} and the techniques in \cite{Chu21} and \cite{Lu23}.
\end{remark}
%For complex Hessian equations, some good third order terms will be lost due to complex conjugacy. Since $\delta_0$ can be small, the term $\delta_0\sum_{i>l}\frac{\sigma_k^{ii}\xi_i^2}{\lambda_1\sigma_k}$ in \eqref{Lu} will be easier to control than the term $2\sum_{i>1}\frac{\sigma_k^{ii}\xi_i^2}{(\lambda_1+A+1)\sigma_k}$ in \eqref{Zhang}. However, $\frac{\xi_1^2}{\lambda_1^2}$ in \eqref{Lu} will produce the term $\sigma_k(\lambda|1)$ in second order estimates, which is hard to handle without the assumption $\chi\in \Gamma_{k+1}(M)$. In contrast, the term $\frac{\sigma_k^{11}\xi_1^2}{\lambda_1\sigma_k}$ in inequality \eqref{Zhang} can avoid this trouble. Hence the key is to control $2\sum_{i>1}\frac{\sigma_k^{ii}\xi_i^2}{(\lambda_1+A+1)\sigma_k}$ in \eqref{Zhang}. Our main idea is optimizing the coefficient $2$ to $1-\varepsilon_0$ for some small $\varepsilon_0>0$ and establishing second order estimates for equation \eqref{K-eq} based on the modified concavity inequality. As the techniques used in \cite{Chu21} and \cite{Lu23}, our test function utilizes the largest eigenvalue of $\chi$, which will make the proof more straightforward and simpler.

The organization of the paper is as follows. In Section 2 we start with some properties for $k$-th elementary symmetric function and some basic calculations for complex Hessian equations on Hermitian manifolds. In Section 3 we prove the crucial concavity inequality. In Section 4 we give the proof of Theorem \ref{main}.

\section{preliminaries}
For $\lambda=(\lambda_1,\dots,\lambda_n)\in\mathbb{R}^n$, the $k$-th elementary symmetric function is defined by
\begin{equation*}
\sigma_k(\lambda)= \sum _{1 \le i_1 < i_2 <\cdots<i_k\leq
n}\lambda_{i_1}\lambda_{i_2}\cdots\lambda_{i_k}.
\end{equation*}
We also set $\sigma_0=1$ and $\sigma_k=0$ for $k>n$ or $k<0$. The G$\mathring{\text{a}}$rding's cone is defined by
\begin{equation*}
\Gamma_k  = \{ \lambda  \in \mathbb{R}^n :\sigma _i (\lambda ) >
0,~\forall~ 1 \le i \le k\}.
\end{equation*}
Denote $\sigma_k(\lambda|i)$ be the $k$-th elementary function with $\lambda_i=0$ and $\sigma_k(\lambda|ij)$ be the $k$-th elementary function with $\lambda_i=\lambda_j=0$. We list some properties of
$\sigma_k$ which will be used later.

\begin{proposition}\label{sigma}
Let $\lambda=(\lambda_1,\dots,\lambda_n)\in\mathbb{R}^n$ and $1\leq
k\leq n$, then we have

(1) $\Gamma_1\supset \Gamma_2\supset \cdot\cdot\cdot\supset
\Gamma_n$;

(2) $\sigma_{k-1}(\lambda|i)>0$ for $\lambda \in \Gamma_k$ and
$1\leq i\leq n$;

(3) $\sigma_k(\lambda)=\sigma_k(\lambda|i)
+\lambda_i\sigma_{k-1}(\lambda|i)$ for $1\leq i\leq n$;

(4)
$\sum_{i=1}^{n}\frac{\partial[\frac{\sigma_{k}}{\sigma_{l}}]^{\frac{1}{k-l}}}
{\partial \lambda_i}\geq [\frac{C^k_n}{C^l_n}]^{\frac{1}{k-l}}$ for
$\lambda \in \Gamma_{k}$ and $0\leq l<k$;

(5) $\Big[\frac{\sigma_k}{\sigma_l}\Big]^{\frac{1}{k-l}}$ are
concave in $\Gamma_k$ for $0\leq l<k$;

(6) If $\lambda \in
\Gamma_k$, $\lambda_1\geq \lambda_2\geq \cdot\cdot\cdot\geq \lambda_n$,
then $\sigma_{k-1}(\lambda|1)\leq \sigma_{k-1}(\lambda|2)\leq
\cdot\cdot\cdot\leq \sigma_{k-1}(\lambda|n)$;

(7)
$\sum_{i=1}^{n}\sigma_{k-1}(\lambda|i)=(n-k+1)\sigma_{k-1}(\lambda)$;

(8)
For $\lambda\in \Gamma_k$ and $\lambda_1\geq\cdots\geq\lambda_n$, $\lambda_1\sigma_{k-1}(\lambda|1)\geq\frac{k}{n}\sigma_k(\lambda)$.
\end{proposition}

\begin{proof}
All the properties are well known. For example, see Chapter XV in
\cite{Li96} or \cite{Hui99} for (1), (2), (3), (6) and (7); see Lemma 2.2.19 in \cite{Ger06} for (4); see
\cite{CNS85} and \cite{Li96} for (5); see Lemma 2 in \cite{D21} for (8).
\end{proof}

The generalized Newton-MacLaurin inequality is as follows.
\begin{proposition}\label{NM}
For $\lambda \in \Gamma_m$ and $m > l \geq 0$, $ r > s \geq 0$, $m
\geq r$, $l \geq s$, we have
\begin{align}
\Bigg[\frac{{\sigma _m (\lambda )}/{C_n^m }}{{\sigma _l (\lambda
)}/{C_n^l }}\Bigg]^{\frac{1}{m-l}} \le \Bigg[\frac{{\sigma _r
(\lambda )}/{C_n^r }}{{\sigma _s (\lambda )}/{C_n^s
}}\Bigg]^{\frac{1}{r-s}}. \notag
\end{align}
\end{proposition}
\begin{proof}
See \cite{S05}.
\end{proof}

To be more clear, we shall rewrite equation \eqref{K-eq} in a local coordinate chart. Let $\lambda\{a_{i\overline{j}}\}$ be the eigenvalues of a Hermitian matrix $\{a_{i\overline{j}}\}$. Define $\sigma_k(a_{i\overline{j}})=\sigma_k(\lambda\{a_{i\overline{j}}\})$. This definition can be naturally extended to complex manifolds. Let $\mathcal{A}^{1,1}(M)$ be the space of real smooth $(1,1)$-forms on Hermitian manifold $(M,\omega)$. For any $\chi\in\mathcal{A}^{1,1}(M)$, we write in a local coordinate chart $(z^1,\cdots,z^n)$
$$\chi=\sqrt{-1}\chi_{i\overline{j}}dz^i\wedge d\overline{z}^j,$$
and define
$$\sigma_k(\chi)=C_n^k\frac{\chi^k\wedge \omega^{n-k}}{\omega^n}.$$
In particular, in a local normal coordinate system $\omega_{i\overline{j}}=\delta_{ij}$, equation \eqref{K-eq} can be rewritten by
\begin{eqnarray}\label{K-eq-1}
\sigma_k(\chi)=\sigma_k(\chi'_{i\overline{j}}+u_{i\overline{j}}+a_iu_{\overline{j}}+a_{\overline{j}}u_i)=\psi(z, Du, u).
\end{eqnarray}

\begin{lemma}
If $F(A)= f(\lambda_1, \ldots, \lambda_n)$ is a symmetric function of the eigenvalues of a Hermitian matrix $A=\{a_{i \ov j}\}$, then at a diagonal matrix $A$ with distinct eigenvalues, we have
\begin{align}
\label{symmetric func 1th deriv} F^{i\ov j} =&\ \delta_{ij} f_i,\\
\label{symmetric func 2th deriv} F^{i\ov j, r\ov s} w_{i\ov j k} w_{r\ov s \ov k} =&\ \sum_{i,j} f_{ij} w_{i\ov i k} w_{j \ov j \ov k} + \sum_{p\neq q}\frac{f_p - f_q}{\lambda_p-\lambda_q} | w_{p\ov q k}|^2,
\end{align}
where $F^{i\ov j}=\frac{\partial F}{\partial a_{i \ov j}}$,
$F^{i\ov j,r\ov s}=\frac{\partial^2 F}{\partial a_{i \ov j}\partial a_{r \ov s}}$ and $w_{i\ov jk}$ is an arbitrary tensor.
\end{lemma}
\begin{proof}
See \cite{Ba84,D21}.
\end{proof}

In local complex coordinates $\left(z_{1}, \cdots, z_{n}\right)$, the subscripts of a function $u$ always denote the covariant derivatives of $u$ with respect to $\omega$ in the directions of the local frame $\left(\partial / \partial z^{1}, \cdots, \partial / \partial z^{n}\right)$. Namely,
\begin{equation}\nonumber
u_{i}=D_{i}u=D_{\partial / \partial z^{i}} u, \; u_{i \overline{j}}=D_{\partial / \partial \overline{z}^{j}} D_{\partial / \partial z^{i}} u, \; u_{i \overline{j} \ell}=D_{\partial / \partial z^{\ell}} D_{\partial / \partial \overline{z}^{j}} D_{\partial / \partial z^{i}} u.
\end{equation}
Next, we list some well-known results on Hermitian manifolds.
\begin{lemma}
\begin{equation}\label{order}
\begin{aligned}
u_{pj}= &u_{jp}+T^k_{pj}u_k,  \;\;\\
u_{i \overline{j} \ell}= &\  u_{i \ell \overline{j}}-u_{p} R_{\ell \overline{j} i}{}^p, \;\;\\
u_{p \overline{j} \overline{m}}= &\  u_{p \overline{m} \overline{j}}-\overline{T_{m j}^{q}} u_{p \overline{q}},\;\;\\
u_{i \overline{q} \ell}= &\  u_{\ell \ov q i}-T_{\ell i}^{p} u_{p \overline{q}},
\end{aligned}
\end{equation}
\end{lemma}

\begin{equation}\label{4d}
u_{i \overline{j} \ell \overline{m}} =  u_{\ell \overline{m} i \overline{j}}
+u_{p \overline{j}} R_{\ell \overline{m} i}{}^{p}
 - u_{p \overline{m}} R_{i \overline{j} \ell}{}^{p}
- T_{\ell i}^{p} u_{p \overline{m} \overline{j}}
-\overline{T_{m j}^{q}} u_{\ell \overline{q} i}
 -T_{i \ell}^{p} \overline{T_{m j}^{q}} u_{p \overline{q}}.
\end{equation}
\begin{proof}
See \cite{TW19}.
\end{proof}

Denote
\begin{equation}
\sigma_{k}^{p \overline{q}}=\frac{\partial \sigma_{k}}{\partial \chi_{p \overline{q}}}, \quad \sigma_{k}^{p \overline{q}, r \overline{s}}=\frac{\partial^{2} \sigma_{k}}{\partial \chi_{p \overline{q}} \partial \chi_{r \overline{s}}}, \quad
\mathcal{F}=\sum_{p} \sigma_{k}^{p \overline{p}}.
\end{equation}
We also use the following notations as in \cite{PPZ19},
\begin{equation}
\begin{aligned}
|D D u|_{\sigma \omega}^{2}=\sigma_{k}^{p \overline{q}} \omega^{m \overline{\ell}} u_{mp} u_{\ov \ell \ov q},\;\;
|D \ov{D} u|_{\sigma \omega}^{2}=\sigma_{k}^{p \overline{q}} \omega^{m \overline{\ell}} u_{p\ov\ell} u_{m\ov q},
\end{aligned}
\end{equation}
and
\begin{equation}
|\eta|_{\sigma}^{2}=\sigma_{k}^{p \overline{q}} \eta_{p} \eta_{\overline{q}},
\end{equation}
for any 1-form $\eta$, where $\{\omega^{m\overline{l}}\}$ is the inverse of the metric $\omega$.

%The following basic calculations will be used in next Section. For convenience, we will use a uniform constant $C$ depending on the known data as in Theorem \ref{main}, but may change from line to line.

The following basic calculations are carried out at a point $z$ on the manifold $M$, and we use coordinates such that at this point $\omega=\sqrt{-1} \sum_{k,\ell} \delta_{k \ell} dz^{k} \wedge d \overline{z}^{\ell}$ and $\{\chi_{i \ov{j}}\}$ is diagonal. Note that $\{\sigma_{k}^{i \overline{j}}\}$ is diagonal when $\{\chi_{i \ov{j}}\}$ is diagonal.
Combining with \eqref{order} and \eqref{4d}, we derive the following two inequalities which will be used later.
\begin{lemma}
Let $\lambda=(\lambda_1,\cdots,\lambda_n)$ be the eigenvalues of $\{\chi_{i \ov{j}}\}$ with the ordering $\lambda_1\geq\lambda_2\geq\cdots\geq\lambda_n$, then
\begin{equation}
\begin{aligned}\label{equ}
& \sigma_{k}^{p \overline{q}} D_{\overline{q}} D_{p} \chi_{i \overline{j}}\\
 \geq &\ -\sigma_{k}^{p \overline{q}, r \overline{s}} D_{\ov{j}} \chi_{r \overline{s}} D_{i} \chi_{p \overline{q}}+\sum_\ell \psi_{v_\ell} D_{\ell}\chi_{i \ov{j}} + \sum_\ell \psi_{\ov{v}_\ell}D_{\ov{\ell}} \chi_{i \ov{j}}\\
&\ -\sigma_{k}^{p \ov{q}} (T_{p i}^{a} u_{a \overline{q} \overline{j}} +\overline{T_{q j}^{a}} u_{p \overline{a} i})-C(1+|DDu|^{2}+|D \ov{D}u|^{2}+\lambda_{1} \mathcal{F}+\lambda_{1})\\
&\ +\sigma_k^{p\overline{q}}(a_iu_{\overline{j}p\overline{q}}+a_{\overline{j}}u_{ip\overline{q}}
-a_pu_{\overline{q}i\overline{j}}-a_{\overline{q}}u_{pi\overline{j}})\\
&\ +\sigma_k^{p\overline{q}}(a_{ip}u_{\overline{j}\overline{q}}+a_{\overline{j}\overline{q}}u_{ip}-a_{pi}u_{\overline{q}\overline{j}}
-a_{\overline{q}\overline{j}}u_{pi}),
\end{aligned}
\end{equation}
\begin{equation}
\begin{aligned}\label{derivative of Du}
\sigma_{k}^{p \overline{q}}|D u|_{p \overline{q}}^{2}
\geq 2 \operatorname{Re}\left\{\sum_{p, m}\left(u_{\overline{p}}u_{mp}+u_p u_{m\overline{p}}\right) \psi_{v_{m}}\right\}
+\frac{1}{2}|D D u|_{\sigma \omega}^{2}+\frac{1}{2}|D \ov{D} u|_{\sigma \omega}^{2}-C-C\mathcal{F}.
\end{aligned}
\end{equation}
\end{lemma}
\begin{proof}
See the proof of (19) and (23) in Section 2 of \cite{D21}.
\end{proof}
\section{A concavity inequality}
%In this section, we slightly modify the inequality \eqref{Zhang} in Zhang \cite{Z24}, and derive the following concavity inequality which will be suitable for complex Hessian equations based on the semi-convex assumption.
In this section, we provide a proof of the concavity inequality in the complex case. Since it is similar to the real case in \cite{Z24}, we mainly highlight the differences.

\emph{Proof of Lemma \ref{key lemma}.}
%The proof is similar to \cite{Z24}, we mainly outline the changes.
By (3.27) in \cite{Z24}, we have for $K\geq(k+1)^2$,
\begin{equation}\label{iq c0}
-\sum_{p\neq q}\frac{\sigma_k^{p\overline{p},q\overline{q}}\omega_{p\overline{p}j}\omega_{q\overline{q}\overline{j}}}{\sigma_k}+K\frac{|D_j\sigma_k|^2}{\sigma_k^2}
\geq\frac{(k+2)^2}{(k+1)(k+3)}\frac{|\omega_{1\overline{1}j}|^2}{\lambda_1^2}-\frac{C}{\lambda_1^{\frac{1}{k-1}}}\sum_{i>1}
\frac{\s_k^{i\overline{i}}|\omega_{i\overline{i}j}|^2}{\lambda_1\s_k},
\end{equation}
where $C$ is a positive constant depending on $n,k,\sigma_k,A$.
Recall that
\begin{equation}\label{sk}
    \s_k=\lambda_1\sigma_k^{1\overline{1}}+\s_k(\lambda|1).
\end{equation}
We note that if
\begin{equation}
    \s_k(\lambda|1)\geqslant-\frac{\s_k}{2(k+2)^2-1},
\end{equation}
then \eqref{iq c0} implies the inequality \eqref{key} by choosing $\lambda_1\geq\left(\frac{C}{1-\varepsilon_0}\right)^{k-1}$.

Denote $c_0=\frac{1}{2(k+2)^2-1}$ and next we only need to consider the case
\begin{align}\label{sk1 c0}
    \s_k(\lambda|1)<-c_0\s_k.
\end{align}
Under this assumption, refer to the deduction of (3.73) in \cite{Z24}, we derive
\begin{eqnarray}
  \nonumber&&-\sum_{p\neq q}\frac{\sigma_k^{p\overline{p},q\overline{q}}\omega_{p\overline{p}j}\omega_{q\overline{q}\overline{j}}}{\sigma_k}+K\frac{|D_j\sigma_k|^2}{\sigma_k^2}
  +(1-\varepsilon_0)\sum_{i>1}\frac{\sigma_k^{i\overline{i}}|\omega_{i\overline{i}j}|^2}{\lambda_1\sigma_k}\\
 \nonumber &\geq& (k+1-l)(1-\varepsilon_0-\frac{C_1}{M})(1-\frac{C_1}{M})(1+a)^2\frac{\lambda_1\sigma_k^{1\overline{1}}-\sigma_k}{\lambda_1^2\sigma_k}|\omega_{1\overline{1}j}|^2
  +\frac{(k+2)^2}{(k+1)(k+3)}\frac{|\omega_{1\overline{1}j}|^2}{\lambda_1^2}\\
  \nonumber&&+\frac{2a^2|\omega_{1\overline{1}j}|^2}{\lambda_1}\left(1-\frac{C_2}{M}\right)\frac{\lambda_1\sigma_k^{1\overline{1}}
  -\sigma_k}{\lambda_1\sigma_k}-C_3(a^2+(1-\varepsilon_0)(1+a)^2)\frac{\sigma_k^{1\overline{1}}|\omega_{1\overline{1}j}|^2}{M\lambda_1}\\
 \nonumber &\geq&\frac{2(a^2+(1-\varepsilon_0)(1+a)^2)|\omega_{1\overline{1}j}|^2}{\lambda_1}\left(1-\frac{C_2}{(1-\varepsilon_0)M}\right)
  \frac{\lambda_1\sigma_k^{1\overline{1}}-\sigma_k}{\lambda_1\sigma_k}\\
  \nonumber&&+\frac{(k+2)^2}{(k+1)(k+3)}\frac{|\omega_{1\overline{1}j}|^2}{\lambda_1^2}
  -C_3(a^2+(1-\varepsilon_0)(1+a)^2)\frac{\sigma_k^{1\overline{1}}|\omega_{1\overline{1}j}|^2}{M\lambda_1}\\
 \nonumber &\geq&(2-\varepsilon_0)\frac{(a^2+(1-\varepsilon_0)(1+a)^2)|\omega_{1\overline{1}j}|^2}{\lambda_1}
  \frac{\lambda_1\sigma_k^{1\overline{1}}-\sigma_k}{\lambda_1\sigma_k}
  +\frac{(k+2)^2}{(k+1)(k+3)}\frac{|\omega_{1\overline{1}j}|^2}{\lambda_1^2}\\
  \nonumber&&+\frac{(a^2+(1-\varepsilon_0)(1+a)^2)|\omega_{1\overline{1}j}|^2}{\lambda_1}\left[
  \frac{\varepsilon_0}{2}\frac{\lambda_1\sigma_k^{1\overline{1}}-\sigma_k}{\lambda_1\sigma_k}-\frac{C'_3\sigma_k^{1\overline{1}}}{\sigma_kM}\right],
\end{eqnarray}
by assuming $\lambda_1>M^k$ and $M$ large enough such that $1-\frac{C_2}{(1-\varepsilon_0)M}\geq1-\frac{\varepsilon_0}{4}$. Here $C_1,C_2,C_3,C'_3$ are positive constants depending on $n,k,\sigma_k,A$.

Since under the assumption $\s_k(\lambda|1)<-c_0\s_k$, \eqref{sk} implies that
\begin{align*}
  \frac{\s_k^{1\overline{1}}}{\s_k}= \frac{1}{\lambda_1}-\frac{\s_k(\lambda|1)}{\lambda_1\s_k}> \frac{c_0+1}{\lambda_1}.
\end{align*}
Then by choosing $M\geq\frac{2C'_3(c_0+1)}{c_0\varepsilon_0}$, we get
\begin{eqnarray*}
 \frac{\varepsilon_0}{2}\frac{\lambda_1\sigma_k^{1\overline{1}}-\sigma_k}{\lambda_1\sigma_k}-\frac{C'_3\sigma_k^{1\overline{1}}}{\sigma_kM}
  &=& \frac{\sigma_k^{1\overline{1}}}{\sigma_k}\left(\frac{\varepsilon_0}{2}-\frac{C'_3}{M}\right)-\frac{\varepsilon_0}{2\lambda_1}\\
  &>&\frac{1}{\lambda_1}\left[\left(\frac{\varepsilon_0}{2}-\frac{C'_3}{M}\right)(c_0+1)-\frac{\varepsilon_0}{2}\right]\geq 0.
\end{eqnarray*}
Thus
\begin{align}\label{im}
   &-\sum_{p\neq q}\frac{\sigma_k^{p\overline{p},q\overline{q}}\omega_{p\overline{p}j}\omega_{q\overline{q}\overline{j}}}{\sigma_k}+K\frac{|D_j\sigma_k|^2}{\sigma_k^2}
  +(1-\varepsilon_0)\sum_{i>1}\frac{\sigma_k^{i\overline{i}}|\omega_{i\overline{i}j}|^2}{\lambda_1\sigma_k}\\
 \nonumber \geq&(2-\varepsilon_0)\frac{(a^2+(1-\varepsilon_0)(1+a)^2)|\omega_{1\overline{1}j}|^2}{\lambda_1}
  \frac{\lambda_1\sigma_k^{1\overline{1}}-\sigma_k}{\lambda_1\sigma_k}
  +\frac{(k+2)^2}{(k+1)(k+3)}\frac{|\omega_{1\overline{1}j}|^2}{\lambda_1^2}\\
  \nonumber=&(2-\varepsilon_0)(a^2+(1-\varepsilon_0)(1+a)^2)\frac{\sigma_k^{1\overline{1}}|\omega_{1\overline{1}j}|^2}{\lambda_1\sigma_k}\\
\nonumber &+\left(\frac{(k+2)^2}{(k+1)(k+3)}-(2-\varepsilon_0)[a^2+(1-\varepsilon_0)(1+a)^2]\right)\frac{|\omega_{1\overline{1}j}|^2}{\lambda_1^2}.
 \end{align}
If $\frac{(k+2)^2}{(k+1)(k+3)}\geq(2-\varepsilon_0)[a^2+(1-\varepsilon_0)(1+a)^2]$, \eqref{im} implies that
\begin{eqnarray*}
&&-\sum_{p\neq q}\frac{\sigma_k^{p\overline{p},q\overline{q}}\omega_{p\overline{p}j}\omega_{q\overline{q}\overline{j}}}{\sigma_k}+K\frac{|D_j\sigma_k|^2}{\sigma_k^2}
  +(1-\varepsilon_0)\sum_{i>1}\frac{\sigma_k^{i\overline{i}}|\omega_{i\overline{i}j}|^2}{\lambda_1\sigma_k}\\
  &\geq&(2-\varepsilon_0)(a^2+(1-\varepsilon_0)(1+a)^2)\frac{\sigma_k^{1\overline{1}}|\omega_{1\overline{1}j}|^2}{\lambda_1\sigma_k}\\
  &\geq&(1-\varepsilon_0)\frac{\sigma_k^{1\overline{1}}|\omega_{1\overline{1}j}|^2}{\lambda_1\sigma_k}.
\end{eqnarray*}
In the last inequality, we use the fact that
\begin{eqnarray*}
  a^2+(1-\varepsilon_0)(1+a)^2&=&(2-\varepsilon_0)\left[\left(a+\frac{1-\varepsilon_0}{1+\varepsilon_0}\right)^2-\frac{(1-\varepsilon_0)^2}
  {(2-\varepsilon_0)^2}+\frac{1-\varepsilon_0}{2-\varepsilon_0}\right] \\
  &\geq& 1-\varepsilon_0-\frac{(1-\varepsilon_0)^2}{2-\varepsilon_0}=\frac{1-\varepsilon_0}{2-\varepsilon_0}.
\end{eqnarray*}
If $\frac{(k+2)^2}{(k+1)(k+3)}<(2-\varepsilon_0)[a^2+(1-\varepsilon_0)(1+a)^2]$, \eqref{im} implies that
\begin{eqnarray*}
&&-\sum_{p\neq q}\frac{\sigma_k^{p\overline{p},q\overline{q}}\omega_{p\overline{p}j}\omega_{q\overline{q}\overline{j}}}{\sigma_k}+K\frac{|D_j\sigma_k|^2}{\sigma_k^2}
  +(1-\varepsilon_0)\sum_{i>1}\frac{\sigma_k^{i\overline{i}}|\omega_{i\overline{i}j}|^2}{\lambda_1\sigma_k}\\
  &>&\left(1+\frac{1}{(k+1)(k+3)}\right)\frac{\sigma_k^{1\overline{1}}|\omega_{1\overline{1}j}|^2}{\lambda_1\sigma_k}
  >\left(1-\varepsilon_0\right)\frac{\sigma_k^{1\overline{1}}|\omega_{1\overline{1}j}|^2}{\lambda_1\sigma_k}.
\end{eqnarray*}
Hence we obtain \eqref{key} and complete the proof.

\section{Proof of Theorem \ref{main}}

In this section, we prove Theorem \ref{main}.
Consider the auxiliary function
\begin{eqnarray*}
Q=\log \lambda_1+\varphi(|D u|^2)+\phi(u),
\end{eqnarray*}
where $\lambda_1$ is the largest eigenvalue of the matrix $\{\chi_{i\overline{j}}\}$.

Define
$$\varphi(s)=e^{Ns}, \quad 0\leq s\leq S-1,$$
and
$$\phi(t)=e^{\Lambda(-t+T)}, \quad -T+1\leq t\leq T-1.$$
Here we set
$$S=\sup_M|D u|^2+1, \quad T=\sup_M|u|+1,$$
and $\Lambda, N>1$ are large constants to be determined later. Clearly, $\varphi, \phi$ satisfy
$$\varphi'=N\varphi>0, \quad \varphi''=\frac{(\varphi')^2}{\varphi},$$
$$\phi'=-\Lambda\phi<0, \quad \phi''=\frac{(\phi')^2}{\phi},$$
and
\begin{equation}\label{ine}
  \varphi''-2\phi''\frac{(\varphi')^2}{(\phi')^2}=\varphi''-2\frac{(\varphi')^2}{\phi}=N^2e^{Ns}-\frac{2N^2e^{2Ns}}{e^{\Lambda(-t+T)}}>0,
\end{equation}
when $\Lambda\gg N>1$.
Suppose $Q$ attains its maximum at $z_0 \in M$. Denote $(\lambda_1,\lambda_2,\cdots,\lambda_n)$ are the eigenvalues of the matrix $\{\chi_{i\overline{j}}\}$ with respect to $\omega$. Without loss of generality, assume $\lambda_1$ has multiplicity $m$, we choose the coordinate system centered at $z_0$ such that $\omega=\sqrt{-1}\delta_{kl}dz^k\wedge d\overline{z}^l$ and $\{\chi_{i\overline{j}}\}$ is diagonal with
$$\lambda_1=\lambda_2=\cdots=\lambda_m\geq\lambda_{m+1}\geq\cdots\geq\lambda_n.$$
We now apply a perturbation argument(see \cite{Chu21}). Let $B$ be a matrix satisfying at $z_0$
$$B_{i\overline{j}}=\delta_{ij}(1-\delta_{1i}),\quad B_{i\overline{j}p}=B_{1\overline{1}i\overline{i}}=0.$$
Define $\widetilde{\chi}_{i\overline{j}}=\chi_{i\overline{j}}-B_{i\overline{j}}$ and denote its eigenvalues by $\widetilde{\lambda}=(\widetilde{\lambda}_1,\widetilde{\lambda}_2,\cdots,\widetilde{\lambda}_n)$. It follows that $\lambda_1\geq\widetilde{\lambda}_1$ near $z_0$ and at $z_0$
$$\widetilde{\lambda}_1=\lambda_1,\quad \widetilde{\lambda}_i=\lambda_i-1 \quad\mbox{for}~i>1.$$
Thus $\widetilde{\lambda}_1>\widetilde{\lambda}_2$ at $z_0$, then $\widetilde{\lambda}_1$ is smooth at $z_0$. We consider the new function
$$\widetilde{Q}=\log \widetilde{\lambda}_1+\varphi(|D u|^2)+\phi(u).$$
It still achieves a local maximum at $z_0$. Then at $z_0$, we have
\begin{equation}\label{Qi}
  0=\widetilde{Q}_i=\frac{\widetilde{\lambda}_{1,i}}{\lambda_1}+\varphi'D_i(|D u|^2)+\phi'u_i,
\end{equation}
and
\begin{eqnarray}\label{Qii}
  \nonumber0\geq\sigma_k^{i\overline{i}}\widetilde{Q}_{i\overline{i}}&=&\sigma_k^{i\overline{i}}(\log \widetilde{\lambda}_1)_{i\overline{i}}+\varphi''\sigma_k^{i\overline{i}}|D_i(|D u|^2)|^2\\
  &&+\varphi'\sigma_k^{i\overline{i}}|Du|_{i\overline{i}}^2
  +\phi''\sigma_k^{i\overline{i}}|u_i|^2+\phi'\sigma_k^{i\overline{i}}u_{i\overline{i}}.
\end{eqnarray}
The following calculations are all at $z_0$. For convenience, we will use a unified notation $C$ to denote a constant depending on $(M, \omega),n,k,\varepsilon,\chi',\psi,a,\sup_M|u|$ and $\sup_M|Du|$.

First, we deal with the term $\sigma_k^{i\overline{i}}(\log\widetilde{\lambda}_1)_{i\overline{i}}$ in \eqref{Qii}.
Thus
$$\widetilde{\lambda}_{1,i}=\frac{\partial\widetilde{\lambda}_1}{\partial \widetilde{\chi}_{p\overline{q}}}D_i\widetilde{\chi}_{p\overline{q}}=\delta_{1p}\delta_{1q}D_i\widetilde{\chi}_{p\overline{q}}
=D_i\widetilde{\chi}_{1\overline{1}}=D_i\chi_{1\overline{1}},$$
and
\begin{eqnarray*}
\widetilde{\lambda}_{1,i\overline{i}}&=&\frac{\partial\widetilde{\lambda}_1}{\partial\widetilde{\chi}_{p\overline{q}}}D_{\overline{i}}
D_i\widetilde{\chi}_{p\overline{q}}
+\frac{\partial^2\widetilde{\lambda}_1}{\partial\widetilde{\chi}_{p\overline{q}}
\partial\widetilde{\chi}_{r\overline{s}}}D_i\widetilde{\chi}_{p\overline{q}}D_{\overline{i}}\widetilde{\chi}_{r\overline{s}}\\
&=&\delta_{1p}\delta_{1q}D_{\overline{i}}D_i\widetilde{\chi}_{p\overline{q}}+\left[(1-\delta_{1p})\frac{\delta_{1q}\delta_{1r}\delta_{ps}}
{\widetilde{\lambda}_1-\widetilde{\lambda}_p}+(1-\delta_{1r})\frac{\delta_{1s}\delta_{1p}\delta_{qr}}
{\widetilde{\lambda}_1-\widetilde{\lambda}_r}\right]D_i\widetilde{\chi}_{p\overline{q}}D_{\overline{i}}\widetilde{\chi}_{r\overline{s}}\\
&=&D_{\overline{i}}D_i\widetilde{\chi}_{1\overline{1}}+\sum_{p>1}\frac{|D_i\widetilde{\chi}_{p\overline{1}}|^2
+|D_i\widetilde{\chi}_{1\overline{p}}|^2}{\widetilde{\lambda}_1-\widetilde{\lambda}_p}
=D_{\overline{i}}D_i\chi_{1\overline{1}}+\sum_{p>1}\frac{|D_i\chi_{p\overline{1}}|^2
+|D_i\chi_{1\overline{p}}|^2}{\lambda_1-\widetilde{\lambda}_p}.
\end{eqnarray*}
Using Cauchy-Schwarz inequality, we have
\begin{equation}\label{cs}
  |\sigma_k^{p\overline{q}}(a_{ip}u_{\overline{j}\overline{q}}+a_{\overline{j}\overline{q}}u_{ip}-a_{pi}u_{\overline{q}\overline{j}}
-a_{\overline{q}\overline{j}}u_{pi})|\leq \frac{1}{4}|D\overline{D}u|^2_{\sigma\omega}+\frac{1}{4}|DDu|^2_{\sigma\omega}+C\mathcal{F}.
\end{equation}
Then combining with \eqref{equ} and \eqref{cs}, we derive
\begin{align}
  \sigma_k^{i\overline{i}}(\log \widetilde{\lambda}_1)_{i\overline{i}}=&\frac{\sigma_k^{i\overline{i}}\widetilde{\lambda}_{1,i\overline{i}}}{\lambda_1}
  -\frac{\sigma_k^{i\overline{i}}|\widetilde{\lambda}_{1,i}|^2}{\lambda_1^2}\\
  \nonumber=&\frac{\sigma_k^{i\overline{i}}D_{\overline{i}}D_i\chi_{1\overline{1}}}{\lambda_1}+\sum_{p>1}\sigma_k^{i\overline{i}}
  \frac{|D_i\chi_{p\overline{1}}|^2+|D_i\chi_{1\overline{p}}|^2}{\lambda_1(\lambda_1-\widetilde{\lambda}_p)}
  -\frac{\sigma_k^{i\overline{i}}|D_i\chi_{1\overline{1}}|^2}{\lambda_1^2}\\
 \nonumber \geq&-\frac{\sigma_k^{i\overline{j},r\overline{s}}D_1\chi_{i\overline{j}}D_{\overline{1}}\chi_{r\overline{s}}}{\lambda_1}
  +\frac{1}{\lambda_1}\left(\sum_{\ell}\psi_{v_{\ell}}D_{\ell}\chi_{1\overline{1}}+\sum_{\ell}\psi_{\overline{v}_{\ell}}
  D_{\overline{\ell}}\chi_{1\overline{1}}\right)\\
  \nonumber&-\frac{1}{\lambda_1}\sigma_k^{i\overline{i}}(T^a_{i1}u_{a\overline{i}\overline{1}}+\overline{T^a_{i1}}
  u_{i\overline{a}1})-\frac{1}{\lambda_1}\left(\frac{1}{4}|DDu|_{\sigma\omega}^2+\frac{1}{4}|D\overline{D}u|_{\sigma\omega}^2\right)\\
 \nonumber &+\frac{1}{\lambda_1}\sigma_k^{i\overline{i}}(a_1u_{\overline{1}i\overline{i}}+a_{\overline{1}}u_{1i\overline{i}}
  -a_iu_{\overline{i}1\overline{1}}-a_{\overline{i}}u_{i1\overline{1}})+\sum_{p>1}\sigma_k^{i\overline{i}}
  \frac{|D_i\chi_{p\overline{1}}|^2+|D_i\chi_{1\overline{p}}|^2}{\lambda_1(\lambda_1-\widetilde{\lambda}_p)}\\
  \nonumber&-\frac{\sigma_k^{i\overline{i}}|D_i\chi_{1\overline{1}}|^2}{\lambda_1^2}-\frac{C}{\lambda_1}(1+|DDu|^2+|D\overline{D}u|^2)-C-C\mathcal{F}.
\end{align}
Using \eqref{derivative of Du} and \eqref{Qi}, we get
\begin{align}\label{la}
  &\frac{1}{\lambda_1}\left(\sum_{\ell}\psi_{v_{\ell}}D_{\ell}\chi_{1\overline{1}}+\sum_{\ell}\psi_{\overline{v}_{\ell}}
  D_{\overline{\ell}}\chi_{1\overline{1}}\right)+\varphi'\sigma_k^{i\overline{i}}|Du|_{i\overline{i}}^2\\
  \nonumber\geq &\frac{1}{\lambda_1}\left(\sum_{\ell}\psi_{v_{\ell}}D_{\ell}\chi_{1\overline{1}}+\sum_{\ell}\psi_{\overline{v}_{\ell}}
  D_{\overline{\ell}}\chi_{1\overline{1}}\right)+\frac{\varphi'}{2}|D D u|_{\sigma \omega}^{2}+\frac{\varphi'}{2}|D \ov{D} u|_{\sigma \omega}^{2}\\
  \nonumber&+2\varphi' \operatorname{Re}\left\{\sum_{p, m}\psi_{v_{m}}\left(u_{\overline{p}}u_{mp}+u_pu_{m\overline{p}}\right) \right\}-C\varphi'-C\varphi'\mathcal{F}\\
  \nonumber\geq&-\phi'\sum_m(\psi_{v_m}u_m+\psi_{\overline{v}_m}u_{\overline{m}})+\frac{\varphi'}{2}|D D u|_{\sigma \omega}^{2}+\frac{\varphi'}{2}|D \ov{D} u|_{\sigma \omega}^{2}-C\varphi'-C\varphi'\mathcal{F}\\
  \nonumber\geq&\frac{\varphi'}{2}|D D u|_{\sigma \omega}^{2}+\frac{\varphi'}{2}|D \ov{D} u|_{\sigma \omega}^{2}
  +C\phi'-C\varphi'-C\varphi'\mathcal{F}.
\end{align}
Since $\lambda_i\geq-A$ for any $i=1,\cdots, n$. Assume that $\lambda_1\geq\max\{1,A\}$ large enough, thus
$$\frac{1}{\lambda_1-\widetilde{\lambda}_i}=\frac{1}{\lambda_1-\lambda_i+1}\geq\frac{1}{3\lambda_1},$$
then by Cauchy-Schwarz inequality, for any $\beta\in(0,1)$,
\begin{align}\label{1}
  &\frac{1}{\lambda_1}\sigma_k^{i\overline{i}}|T^a_{i1}u_{a\overline{i}\overline{1}}+\overline{T^a_{i1}}u_{i\overline{a}1}|\\
  \nonumber=&\frac{2}{\lambda_1}
  \sigma_k^{i\overline{i}}\left|\operatorname{Re}\left\{\overline{T^p_{i1}}u_{i\overline{p}1}\right\}\right|\leq\frac{2}{\lambda_1}
  \sigma_k^{i\overline{i}}\left|\overline{T^p_{i1}}(u_{1\overline{p}i}+T^q_{i1}u_{q\overline{p}})\right|\\
  \nonumber\leq&\frac{2}{\lambda_1}\sigma_k^{i\overline{i}}\left|\overline{T^p_{i1}}D_i\chi_{1\overline{p}}\right|
  +\frac{C}{\lambda_1}(|DDu|_{\sigma\omega}^2+|D\overline{D}u|_{\sigma\omega}^2)+C\mathcal{F}\\
  \nonumber\leq&\frac{\beta}{6}\sum_p\frac{\sigma_k^{i\overline{i}}|D_i\chi_{1\overline{p}}|^2}{\lambda_1^2}+\frac{C}{\beta}\mathcal{F}
  +\frac{C}{\lambda_1}(|DDu|_{\sigma\omega}^2+|D\overline{D}u|_{\sigma\omega}^2)\\
  \nonumber\leq&\frac{\beta}{2}\frac{\sigma_k^{i\overline{i}}|D_i\chi_{1\overline{1}}|^2}{\lambda_1^2}+
  \frac{\beta}{2}\sum_{p>1}\frac{\sigma_k^{i\overline{i}}|D_i\chi_{1\overline{p}}|^2}{\lambda_1(\lambda_1-\widetilde{\lambda}_p)}
  +\frac{C}{\beta}\mathcal{F}+\frac{C}{\lambda_1}(|DDu|_{\sigma\omega}^2+|D\overline{D}u|_{\sigma\omega}^2).
\end{align}
We also have
\begin{align}\label{2}
  & \frac{1}{\lambda_1}\sigma_k^{i\overline{i}}|a_1u_{\overline{1}i\overline{i}}+a_{\overline{1}}u_{1i\overline{i}}
  -a_iu_{\overline{i}1\overline{1}}-a_{\overline{i}}u_{i1\overline{1}}|\\
  \nonumber\leq &\frac{\beta}{6}\left(\frac{\sigma_k^{i\overline{i}}|D_i\chi_{1\overline{i}}|^2}{\lambda_1^2}
  +\frac{\sigma_k^{i\overline{i}}|D_i\chi_{1\overline{1}}|^2}{\lambda_1^2}\right)+\frac{C}{\lambda_1}(|DDu|_{\sigma\omega}^2
  +|D\overline{D}u|_{\sigma\omega}^2)+\frac{C}{\beta}\mathcal{F}\\
 \nonumber \leq&\frac{\beta}{6}\sum_p\frac{\sigma_k^{i\overline{i}}|D_i\chi_{1\overline{p}}|^2}{\lambda_1^2}
  +\frac{C}{\beta}\mathcal{F}+\frac{C}{\lambda_1}(|DDu|_{\sigma\omega}^2+|D\overline{D}u|_{\sigma\omega}^2)\\
 \nonumber \leq&\frac{\beta}{2}\frac{\sigma_k^{i\overline{i}}|D_i\chi_{1\overline{1}}|^2}{\lambda_1^2}
  +\frac{\beta}{2}\sum_{p>1}\frac{\sigma_k^{i\overline{i}}|D_i\chi_{1\overline{p}}|^2}{\lambda_1(\lambda_1-\widetilde{\lambda}_p)}
  +\frac{C}{\beta}\mathcal{F}+\frac{C}{\lambda_1}(|DDu|_{\sigma\omega}^2+|D\overline{D}u|_{\sigma\omega}^2).
\end{align}
Using the critical equation \eqref{Qi}, we obtain
\begin{equation}\label{3}
  \phi''\sigma_k^{i\overline{i}}|u_i|^2\geq\frac{\phi''}{2(\phi')^2}\frac{\sigma_k^{i\overline{i}}|D_i\chi_{1\overline{1}}|^2}{\lambda_1^2}
  -\phi''\left(\frac{\varphi'}{\phi'}\right)^2\sigma_k^{i\overline{i}}|D_i(|Du|^2)|^2,
\end{equation}
and
\begin{align}\label{4}
  \left|\phi'\sigma_k^{i\overline{i}}(a_iu_{\overline{i}}+a_{\overline{i}}u_i)\right| =&2\left|\phi'\sigma_k^{i\overline{i}}\operatorname{Re}\left\{a_iu_{\overline{i}}\right\}\right| \\
 \nonumber =&2\left|\sigma_k^{i\overline{i}}\operatorname{Re}\left\{a_i\left(\frac{D_{\overline{i}}\chi_{1\overline{1}}}{\lambda_1}+\varphi'D_{\overline{i}}
  (|Du|^2)\right)\right\}\right|\\
 \nonumber  \leq&\frac{\phi''}{8(\phi')^2}\sigma_k^{i\overline{i}}\left|\frac{D_{\overline{i}}\chi_{1\overline{1}}}{\lambda_1}+\varphi'D_{\overline{i}}
  (|Du|^2)\right|^2+\frac{C(\phi')^2}{\phi''}\mathcal{F}\\
  \nonumber \leq&\frac{\phi''}{4(\phi')^2}\frac{\sigma_k^{i\overline{i}}|D_{\overline{i}}\chi_{1\overline{1}}|^2}{\lambda_1^2}
  +(\varphi')^2\frac{\phi''}{4(\phi')^2}\sigma_k^{i\overline{i}}|D_i(|Du|^2)|^2+\frac{C(\phi')^2}{\phi''}\mathcal{F}.
\end{align}
Since $\chi'(z,u)\geq \varepsilon\omega$, then
\begin{eqnarray}\label{u}
\nonumber-\sigma_{k}^{p \overline{q}} u_{p \overline{q}}&=&\sigma_{k}^{p \overline{q}}\left(\chi'_{p \overline{q}}+a_pu_{\overline{q}}+a_{\overline{q}}u_p-\chi_{p \overline{q}}\right)\\
&\geq& \varepsilon \mathcal{F}+\sigma_{k}^{p \overline{q}}
(a_pu_{\overline{q}}+a_{\overline{q}}u_p)-k \psi.
\end{eqnarray}
Combining with \eqref{Qii}-\eqref{u}, we get
\begin{align}\label{5}
  0\geq& -\frac{1}{\lambda_1}\sigma_k^{i\overline{j},r\overline{s}}D_1\chi_{i\overline{j}}D_{\overline{1}}\chi_{r\overline{s}}
  +\left(\frac{\varphi'}{2}-\frac{C}{\lambda_1}\right)\left(|DDu|_{\sigma\omega}^2+|D\overline{D}u|_{\sigma\omega}^2\right) \\
  \nonumber&+(1-\beta)\sum_{p>1}\sigma_k^{i\overline{i}}\frac{|D_i\chi_{p\overline{1}}|^2
  +|D_i\chi_{1\overline{p}}|^2}{\lambda_1(\lambda_1-\widetilde{\lambda}_p)}-\left(1+\beta-\frac{\phi''}{4(\phi')^2}\right)
  \frac{\sigma_k^{i\overline{i}}|D_i\chi_{1\overline{1}}|^2}{\lambda_1^2}\\
  \nonumber&+\left(\varphi''-2\phi''\frac{(\varphi')^2}{(\phi')^2}\right)\sigma_k^{i\overline{i}}|D_i(|Du|^2)|^2
  -\frac{C}{\lambda_1}(|DDu|^2+|D\overline{D}u|^2)\\
  \nonumber&+\left(-\varepsilon\phi'-C\varphi'-\frac{C(\phi')^2}{\phi''}-\frac{C}{\beta}\right)\mathcal{F}+C\phi'-C\varphi'-C.
\end{align}
According to \eqref{symmetric func 2th deriv},
\begin{align}\label{6}
  &-\frac{1}{\lambda_1}\sigma_k^{i\overline{j},r\overline{s}}D_1\chi_{i\overline{j}}D_{\overline{1}}\chi_{r\overline{s}}
  +(1-\beta)\sum_{p>1}\sigma_k^{i\overline{i}}\frac{|D_i\chi_{p\overline{1}}|^2
  +|D_i\chi_{1\overline{p}}|^2}{\lambda_1(\lambda_1-\widetilde{\lambda}_p)}\\
  \nonumber\geq&-\frac{1}{\lambda_1}\sum_{p\neq q}
  \sigma_k^{p\overline{p},q\overline{q}}D_1\chi_{p\overline{p}}D_{\overline{1}}\chi_{q\overline{q}}+\frac{1}{\lambda_1}\sum_{p\neq q}\frac{\sigma_k^{p\overline{p}}-\sigma_k^{q\overline{q}}}{\lambda_q-\lambda_p}|D_1\chi_{p\overline{q}}|^2\\
  \nonumber&+(1-\beta)\sum_{p>1}\sigma_k^{1\overline{1}}\frac{|D_1\chi_{p\overline{1}}|^2
  +|D_1\chi_{1\overline{p}}|^2}{\lambda_1(\lambda_1-\widetilde{\lambda}_p)}
  +(1-\beta)\sum_{p>1}\sigma_k^{p\overline{p}}\frac{|D_p\chi_{p\overline{1}}|^2
  +|D_p\chi_{1\overline{p}}|^2}{\lambda_1(\lambda_1-\widetilde{\lambda}_p)}\\
 \nonumber \geq&-\frac{1}{\lambda_1}\sum_{p\neq q}\sigma_k^{p\overline{p},q\overline{q}}D_1\chi_{p\overline{p}}D_{\overline{1}}\chi_{q\overline{q}}
  +\sum_{p>m}\frac{\sigma_k^{p\overline{p}}-\sigma_k^{1\overline{1}}}{\lambda_1(\lambda_1-\lambda_p)}|D_1\chi_{p\overline{1}}|^2\\
  \nonumber&+(1-\beta)\sum_{p>1}\frac{\sigma_k^{1\overline{1}}|D_1\chi_{p\overline{1}}|^2}{\lambda_1(\lambda_1-\widetilde{\lambda}_p)}
  +(1-\beta)\sum_{p>1}\frac{\sigma_k^{p\overline{p}}|D_p\chi_{1\overline{p}}|^2}{\lambda_1(\lambda_1-\widetilde{\lambda}_p)}.
\end{align}
On the one hand, by Cauchy-Schwarz inequality, for any $\beta\in(0,1)$,
\begin{align}\label{7}
  & \sum_{p>m}\frac{\sigma_k^{p\overline{p}}-\sigma_k^{1\overline{1}}}{\lambda_1(\lambda_1-\lambda_p)}|D_1\chi_{p\overline{1}}|^2
  +(1-\beta)\sum_{p>1}\frac{\sigma_k^{1\overline{1}}|D_1\chi_{p\overline{1}}|^2}{\lambda_1(\lambda_1-\widetilde{\lambda}_p)}
  -(1-3\beta)\sum_{p>1}\frac{\sigma_k^{p\overline{p}}|D_p\chi_{1\overline{1}}|^2}{\lambda_1^2} \\
  \nonumber \geq&(1-\beta)\sum_{p>m}\frac{\sigma_k^{p\overline{p}}-\sigma_k^{1\overline{1}}}{\lambda_1(\lambda_1-\lambda_p)}|D_1\chi_{p\overline{1}}|^2
   +(1-\beta)\sum_{p>1}\frac{\sigma_k^{1\overline{1}}|D_1\chi_{p\overline{1}}|^2}{\lambda_1(\lambda_1-\widetilde{\lambda}_p)}\\
  \nonumber &-(1-3\beta)(1+\beta)\sum_{p>1}\frac{\sigma_k^{p\overline{p}}|D_1\chi_{p\overline{1}}|^2}{\lambda_1^2}
   -C(1-3\beta)\left(1+\frac{1}{\beta}\right)\mathcal{F}\\
  \nonumber \geq&(1-\beta)\left[\sum_{p>m}\frac{\sigma_k^{p\overline{p}}-\sigma_k^{1\overline{1}}}{\lambda_1(\lambda_1-\lambda_p)}
   |D_1\chi_{p\overline{1}}|^2+\sum_{p>1}\frac{\sigma_k^{1\overline{1}}|D_1\chi_{p\overline{1}}|^2}{\lambda_1(\lambda_1-\widetilde{\lambda}_p)}
   -(1-\beta)\sum_{p>1}\frac{\sigma_k^{p\overline{p}}|D_1\chi_{p\overline{1}}|^2}{\lambda_1^2}\right]-\frac{C}{\beta}\mathcal{F}\\
   \nonumber\geq&-\frac{C}{\beta}\mathcal{F}.
\end{align}
In the last inequality, we use the fact
\begin{align*}
   &\sum_{p>m}\frac{\sigma_k^{p\overline{p}}-\sigma_k^{1\overline{1}}}{\lambda_1(\lambda_1-\lambda_p)}
   |D_1\chi_{p\overline{1}}|^2+\sum_{p>1}\frac{\sigma_k^{1\overline{1}}|D_1\chi_{p\overline{1}}|^2}{\lambda_1(\lambda_1-\widetilde{\lambda}_p)}
   -(1-\beta)\sum_{p>1}\frac{\sigma_k^{p\overline{p}}|D_1\chi_{p\overline{1}}|^2}{\lambda_1^2} \\
  =&\sum_{p>m}\frac{\sigma_k^{p\overline{p}}-\sigma_k^{1\overline{1}}}{\lambda_1(\lambda_1-\lambda_p)}
   |D_1\chi_{p\overline{1}}|^2+\sum_{1<i\leq m}\frac{\sigma_k^{1\overline{1}}|D_1\chi_{i\overline{1}}|^2}{\lambda_1(\lambda_1-\widetilde{\lambda}_i)}
   +\sum_{p>m}\frac{\sigma_k^{1\overline{1}}|D_1\chi_{p\overline{1}}|^2}{\lambda_1(\lambda_1-\widetilde{\lambda}_p)}\\
   &-(1-\beta)\sum_{1<i\leq m}\frac{\sigma_k^{1\overline{1}}|D_1\chi_{i\overline{1}}|^2}{\lambda_1^2}
   -(1-\beta)\sum_{p>m}\frac{\sigma_k^{p\overline{p}}|D_1\chi_{p\overline{1}}|^2}{\lambda_1^2}\\
   =&\sum_{1<i\leq m}\frac{\beta\lambda_1+(1-\beta)\widetilde{\lambda}_i}{\lambda_1^2(\lambda_1-\widetilde{\lambda}_i)}\sigma_k^{1\overline{1}}
   |D_1\chi_{i\overline{1}}|^2+\sum_{p>m}\frac{\beta\lambda_1
   +(1-\beta)\lambda_p}{\lambda_1^2(\lambda_1-\lambda_p)}\sigma_k^{p\overline{p}}|D_1\chi_{p\overline{1}}|^2\\
   &+\sum_{p>m}\frac{\widetilde{\lambda}_p-\lambda_p}{\lambda_1(\lambda_1-\lambda_p)(\lambda_1-\widetilde{\lambda}_p)}\sigma_k^{1\overline{1}}
   |D_1\chi_{p\overline{1}}|^2\\
   \geq &\sum_{1<i\leq m}\frac{\beta\lambda_1+(1-\beta)\widetilde{\lambda}_i}{\lambda_1^2(\lambda_1-\widetilde{\lambda}_i)}\sigma_k^{1\overline{1}}
   |D_1\chi_{i\overline{1}}|^2+\sum_{p>m}\frac{\beta\lambda_1+(1-\beta)\widetilde{\lambda}_p}
   {\lambda_1^2(\lambda_1-\widetilde{\lambda}_p)}\sigma_k^{1\overline{1}}|D_1\chi_{p\overline{1}}|^2 \geq0,
\end{align*}
by choosing $\lambda_1\geq\frac{1-\beta}{\beta}(A+1)$ large enough such that $\beta\lambda_1+(1-\beta)\widetilde{\lambda}_i\geq 0$ for any $i$.

On the other hand, according to Lemma \ref{key lemma} and choosing $\varepsilon_0=3\beta$ and $\lambda_1\geq\frac{1-2\beta}{\beta}(A+1)$ large enough, we derive
\begin{align}\label{8}
  & -\frac{1}{\lambda_1}\sum_{p\neq q}\sigma_k^{p\overline{p},q\overline{q}}D_1\chi_{p\overline{p}}D_{\overline{1}}\chi_{q\overline{q}}
  +(1-\beta)\sum_{p>1}\frac{\sigma_k^{p\overline{p}}|D_p\chi_{1\overline{p}}|^2}{\lambda_1(\lambda_1-\widetilde{\lambda}_p)}
  -(1-3\beta)\frac{\sigma_k^{1\overline{1}}|D_1\chi_{1\overline{1}}|^2}{\lambda_1^2}\\
  \nonumber\geq& -K\frac{|D_1\psi|^2}{\lambda_1\sigma_k}+(1-\varepsilon_0)\frac{\sigma_k^{1\overline{1}}|D_1\chi_{1\overline{1}}|^2}{\lambda_1^2}
  -(1-\varepsilon_0)\sum_{p>1}\frac{\sigma_k^{p\overline{p}}|D_1\chi_{p\overline{p}}|^2}{\lambda_1^2}\\
  \nonumber&+(1-\beta)\sum_{p>1}\frac{\sigma_k^{p\overline{p}}|D_p\chi_{1\overline{p}}|^2}{\lambda_1(\lambda_1-\widetilde{\lambda}_p)}
  -(1-3\beta)\frac{\sigma_k^{1\overline{1}}|D_1\chi_{1\overline{1}}|^2}{\lambda_1^2}\\
  \nonumber\geq&-\frac{CK}{\lambda_1}(|DDu|^2+|D\overline{D}u|^2)+(1-\beta)\sum_{p>1}\frac{\sigma_k^{p\overline{p}}|D_p\chi_{1\overline{p}}|^2}
  {\lambda_1(\lambda_1-\widetilde{\lambda}_p)}\\
  \nonumber&-(1-3\beta)(1+\beta)\sum_{p>1}\frac{\sigma_k^{p\overline{p}}|D_p\chi_{1\overline{p}}|^2}
  {\lambda_1^2}-\frac{C}{\beta}\mathcal{F}\\
  \nonumber\geq&\sum_{p>1}\frac{\beta\lambda_1+(1-2\beta)\widetilde{\lambda}_p}{\lambda_1^2(\lambda_1-\widetilde{\lambda}_p)}
  \sigma_k^{p\overline{p}}|D_p\chi_{1\overline{p}}|^2-\frac{CK}{\lambda_1}(|DDu|^2+|D\overline{D}u|^2)-\frac{C}{\beta}\mathcal{F}\\
  \nonumber\geq&-\frac{CK}{\lambda_1}(|DDu|^2+|D\overline{D}u|^2)-\frac{C}{\beta}\mathcal{F}.
\end{align}
Since $\sigma_k^{i\overline{i}}\geq\sigma_k^{1\overline{1}}\geq\frac{k\sigma_k}{n\lambda_1}\geq\frac{C}{\lambda_1}$ for any $i$, then
\begin{equation}\label{9}
  |DDu|_{\sigma\omega}^2+|D\overline{D}u|_{\sigma\omega}^2\geq\frac{C}{\lambda_1}(|DDu|^2+|D\overline{D}u|^2)\geq\frac{C}{\lambda_1}|DDu|^2
  +C\lambda_1.
\end{equation}
Combining with \eqref{ine} and \eqref{5}-\eqref{9}, we obtain
\begin{align}\label{10}
  0\geq & \left(\frac{\phi''}{4(\phi')^2}-4\beta\right)\frac{\sigma_k^{i\overline{i}}|D_i\chi_{1\overline{1}}|^2}{\lambda_1^2}
  +\left(\frac{\varphi'}{2}-\frac{C}{\lambda_1}\right)(|DDu|_{\sigma\omega}^2+|D\overline{D}u|_{\sigma\omega}^2)\\
  \nonumber& -\frac{C(K+1)}{\lambda_1}(|DDu|^2+|D\overline{D}u|^2)+\left(-\varepsilon\phi'-C\varphi'-\frac{C(\phi')^2}{\phi''}
  -\frac{C}{\beta}\right)\mathcal{F}\\
 \nonumber &+C\phi'-C\varphi'-C\\
 %\nonumber \geq&\frac{\varphi'}{4}\left(\frac{C}{\lambda_1}|DDu|^2+C\lambda_1\right)-\frac{C(K+1)}{\lambda_1}(|DDu|^2+|D\overline{D}u|^2)\\
  %\nonumber&+(-\varepsilon\phi'-C\varphi'-C\phi)\mathcal{F}+C\phi'-C\varphi'-C\\
  \nonumber\geq&\left(\frac{C\varphi'}{\lambda_1}-\frac{C(K+1)}{\lambda_1}\right)|DDu|^2+\left(C\varphi'-C(K+1)\right)\lambda_1\\
 \nonumber &+(-\varepsilon\phi'-C\varphi'-C\phi)\mathcal{F}+C\phi'-C\varphi'-C,
\end{align}
by choosing $\varepsilon_0=3\beta=\frac{3}{16e^{\Lambda(-u(z_0)+T)}}=\frac{3\phi''}{16(\phi')^2}(z_0)$ and $\lambda_1\geq\frac{4C}{N}$ large enough such that $\frac{\varphi'}{2}-\frac{C}{\lambda_1}\geq\frac{\varphi'}{4}$. Then choosing $N$ large enough such that $C\varphi'-C(K+1)\geq1$ and $\Lambda\gg N$ such that
$$-\varepsilon\phi'-C\varphi'-C\phi=\varepsilon\Lambda\phi-CN\varphi-C\phi>0.$$
Thus \eqref{10} implies that
$$\lambda_1\leq -C\phi'+C\varphi'+C.$$
The proof of Theorem \ref{main} is completed.

%%%%%%%%%%%%%%%%%%%%%%%%%%%%%%%%%%%%%%%%%%%
 \bibliographystyle{siam}
% \bibliography{article}

\end{document}